\documentclass[twoside,12pt]{article}

\usepackage{amsmath}
\usepackage{amssymb}
\usepackage{pxfonts}
\usepackage{graphicx}
\usepackage{eucal}
\usepackage{mathrsfs}
\usepackage{theorem}
\usepackage{pifont}
\usepackage{amscd}

\usepackage{sectsty}
\usepackage{pseudocode}

\usepackage{color}
\usepackage{fancyhdr}
\usepackage{framed}
\usepackage[all]{xy}
\usepackage{hyperref}

\definecolor{shadecolor}{rgb}{0.8,0.8,0.8}

\theoremheaderfont{\fontfamily{pzc}\bfseries\large}
\newtheorem{theorem}{Theorem}[section]

\newtheorem{lemma}[theorem]{Lemma}
\newtheorem{proposition}[theorem]{Proposition}
\newtheorem{corollary}[theorem]{Corollary}
\newtheorem{definition}[theorem]{Definition}

\usepackage[all]{xy}
\usepackage{tikz} 
\usepackage{tkz-euclide}
\usetkzobj{all}
\usepackage{pgfplots}

\newcommand{\btkz}{\begin{tikzpicture}}
\newcommand{\etkz}{\end{tikzpicture}}

\pgfdeclareradialshading{ballshading}{
 \pgfpoint{-10bp}{10bp}}
 {color(0bp)=(gray!10!white); 
  color(9bp)=(gray!35!white);
  color(18bp)=(gray!30!black!40); 
  color(25bp)=(gray!20!black!40); 
  color(50bp)=(black)}

{\theorembodyfont{\rmfamily}

\newenvironment{proof}{{\flushleft \emph{Proof}:}}{\hfill\ding{110}}
\newenvironment{comment}{{\flushleft \fontfamily{pzc}\bfseries\large Comment:}}{}

\newcommand{\secref}[1]{Section~\ref{#1}}
\newcommand{\figref}[1]{Figure~\ref{#1}}

\newcommand{\defref}[1]{Definition~\ref{#1}}

\newcommand{\propref}[1]{Proposition~\ref{#1}}
\newcommand{\lemref}[1]{Lemma~\ref{#1}}
\newcommand{\corref}[1]{Corollary~\ref{#1}}

\newcommand{\cof}{\operatorname{cof}}
\newcommand{\Cof}{\operatorname{Cof}}
\newcommand{\Det}{\operatorname{Det}}

\newcommand{\Vol}{\text{Vol}}
\newcommand{\Volume}{\text{Vol}}

\newcommand{\M}{\mathcal{M}}
\newcommand{\MOne}{\M_1}
\newcommand{\MTwo}{\M_2}
\newcommand{\gOne}{\g_1}
\newcommand{\gTwo}{\g_2}
\newcommand{\dVolOne}{d\text{Vol}_1}
\newcommand{\dVolTwo}{d\text{Vol}_2}

\newcommand{\calX}{\mathcal{X}}

\newcommand{\Ga}{\Gamma}
\renewcommand{\div}{\operatorname{div}}
\newcommand{\piola}{\operatorname{Piola}}

\newcommand{\tr}{\operatorname{tr}}
\newcommand{\trg}{\operatorname{tr}_{\g}}
\newcommand{\id}{\operatorname{I}}
\newcommand{\g}{\mathfrak{g}}
\newcommand{\h}{\mathfrak{h}}
\newcommand{\sgn}{\text{sgn}}
\newcommand{\R}{\mathbb{R}}
 
\newcommand{\pl}{\partial}
\renewcommand{\M}{\mathcal{M}}

\newcommand{\til}{\tilde}

\newcommand{\brk}[1]{\left(#1\right)}          
\newcommand{\Average}[1]{\left\langle#1\right\rangle}      

\newcommand{\deriv}[2]{\frac{d#1}{d#2}}
\newcommand{\pd}[2]{\frac{\partial#1}{\partial#2}}

\newcommand{\beq}{\begin{equation}}
\newcommand{\eeq}{\end{equation}}

\newcommand{\Hom}{\operatorname{Hom}}

\newcommand{\IP}[2]{\Average{#1,#2}}
\newcommand{\innerp}[1]{\left\langle#1\right\rangle}

\newcommand{\ga}{\gamma}

\newcommand{\ep}{\epsilon}

\makeatletter
\newcommand{\extp}{\@ifnextchar^\@extp{\@extp^{\,}}}
\def\@extp^#1{\mathop{\bigwedge\nolimits^{\!#1}}}
\makeatother

\numberwithin{equation}{section}

\begin{document}
\title{A geometric perspective on the Piola identity in Riemannian settings}
\author{Raz Kupferman \footnote{Institute of Mathematics, The Hebrew University.} \,and Asaf Shachar\footnote{Institute of Mathematics, The Hebrew University.}}
\date{}
\maketitle

\begin{abstract}
The Piola identity $\div \cof \nabla f=0$ is a central result in the mathematical theory of elasticity. We prove a generalized version of the Piola identity for mappings between Riemannian manifolds, using two approaches, based on different interpretations of the cofactor of a linear map: one follows the lines of the classical Euclidean derivation  and the other is based on a variational interpretation via Null-Lagrangians. In both cases, we first review the Euclidean case before proceeding to the general Riemannian setting. 
\end{abstract}

\tableofcontents


\section{Introduction}

The Piola identity is a classical result in mathematical elasticity, stating the following:

\begin{quote}
Let $\Omega \subseteq \R^d$ be an open domain and let  $f:\Omega \to \R^d$ be a $C^2$-map. Then,
\beq
\label{eq:div_cof_grad_classic}
\div \cof \nabla f=0,
\eeq
where the cofactor of a matrix is the transpose of its adjugate, and the divergence of the matrix-valued function $\cof \nabla f$ is taken row-by-row.
\end{quote}

Equation \eqref{eq:div_cof_grad_classic} can be proved by a direct calculation;  see e.g. \cite[Ch.~8.1.4.b]{Eva98} and \cite[p.~39]{Cia88}. In essence, the analytical derivation boils down to the commutation of mixed partial derivatives. The downside of this ``proof by computation" is that it does not provide any insights on why does this specific combination of second derivatives vanish. In particular, the cofactor of the gradient of a map has a geometric interpretation as the action of that map on $(d-1)$-dimensional surface elements. Thus, one would hope for a more geometric interpretation of the Piola identity \eqref{eq:div_cof_grad_classic}.

A classical geometric derivation of the Piola identity can be found in the mechanical literature \cite[p.~310]{Ste15}. 
Let $x\in \Omega$, and consider a $d$-dimensional ball $B = B_r(x)$ of radius $r$ centered at $x$.
Denote by $\hat{N}\in\R^d$ the unit normal of $\partial B$ and by $dA$ its surface form. If $f$ is smooth and $\nabla f(x)$ is invertible, then for $r$ small enough, $f$ embeds $B$ in $\R^d$ so that $f(\partial B)$ is a topological sphere.  Denote by $\hat{n}\in\R^d$ the unit normal of $f(\partial B)$ and by $da$ its surface form (see \figref{fig:peppe}).

It is an immediate consequence of the divergence theorem that
\beq
\oint_{f(\partial B)} \hat{n}\, da = 0,
\label{eq:peppe1}
\eeq
which is an identity in $\R^d$.
Pulling back \eqref{eq:peppe1} with $f$ and using a well-known property of the cofactor of $\nabla f$, one obtains
\[
0 = \oint_{f(\partial B)} \hat{n}\, da = \int_{\partial B} \cof \nabla f(\hat{N})\, dA.
\]
Applying the divergence theorem (row-by-row),
\[
\int_{B} \div \cof \nabla f \, dV = 0,
\]
where $dV$ is the volume element in $\Omega$.
Letting $r\to0$, we obtain the desired result. Note that this more geometric proof requires $f$ to be a local diffeomorphism, a condition which is not necessary for \eqref{eq:div_cof_grad_classic} to hold.

\begin{figure}
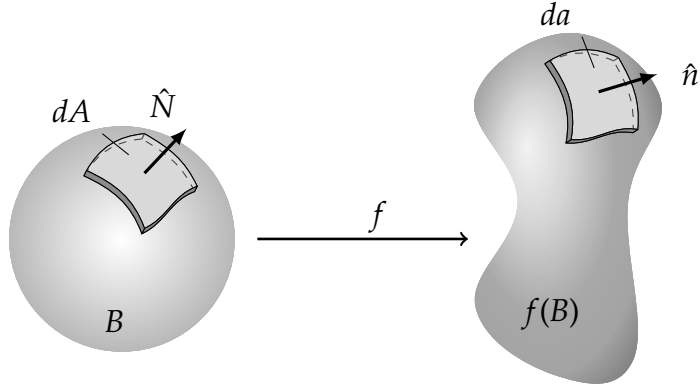

\begin{center}
\btkz
\node[circle,shading=ball, outer color=gray!60!black!40,inner color=white,minimum width=3cm] (ball) at (3,5) {};

\begin{scope}[shift={(7,2)},scale=0.7]
\coordinate (K) at (3,1);
\shade[shading=ballshading] (K) plot [smooth cycle,tension=0.7] coordinates {(1,3) (1.8,5) (1,7) (2.8,8.2) (4.5,7) (3.9,5) (4,2) (2.5,1.7)};
\end{scope}

\tikzset{thin_line/.style={      thin,               solid, color=black}}
\tikzset{dash_line/.style={      thin,              dashed, color=darkgray}}
\tikzset{vect_line/.style={very thick, ->, >=latex,  solid, color=black}}

\begin{scope}[shift={(3.3,5)},scale=0.7]
 \coordinate (a)  at ( 0.00, 0.2);
 \coordinate (am) at (-0.05, 0.1);
 \coordinate (an) at ( 0.45, 0.8);

 \coordinate (b)  at ( 1.00, 1.0);
 \coordinate (bm) at ( 0.95, 0.9);

 \coordinate (c)  at ( 0.00, 2.0);
 \coordinate (cm) at (-0.05, 1.9);

 \coordinate (d)  at (-1.10, 1.3);
 \coordinate (dm) at (-1.15, 1.2);

 \draw[thin_line] (a) to[out= 25, in=200] (b);    
 \draw[thin_line] (b) to[out=115, in=-30] (c);    
 \draw[thin_line] (c) to[out=185, in= 60] (d);    
 \draw[thin_line] (d) to[out=-25, in=110] (a);    

 \draw[fill=gray!30] (a) to[out= 25, in=200] 
                     (b) to[out=115, in=-30] 
                     (c) to[out=185, in= 60] 
                     (d) to[out=-25, in=110] (a);    

 \draw[thin_line] (am) to[out= 25, in=200] (bm);    
 \draw[dash_line] (bm) to[out=115, in=-30] (cm);    
 \draw[dash_line] (cm) to[out=185, in= 60] (dm);    
 \draw[thin_line] (dm) to[out=-25, in=110] (am);    

 \draw[thin_line] (am) -- (a);
 \draw[thin_line] (bm) -- (b);
 \draw[dash_line] (cm) -- (c);
 \draw[thin_line] (dm) -- (d);

 \draw[fill=gray!90] (a)  to[out= 25, in=200] 
                     (b)  -- 
                     (bm) to[out=200, in=25]
                     (am) -- (a);

 \draw[fill=gray!90] (am) --
                     (a)  to[out=110, in=-25]
                     (d)  --
                     (dm) to[out=-25, in=110] (am);

 \coordinate (O)   at ( 0.00, 1.25);
 \coordinate (O1)  at ( 0.16, 1.10);
 \coordinate (O2)  at (-0.10, 1.18);
 \coordinate (P)   at ( 0.80, 2.10);
 \coordinate (h1)  at (-0.80, 2.00);
 \coordinate (h2)  at (-0.30, 1.60);

 \draw[vect_line] (O) -- (P);
 \draw[thin_line] (h1) node[above left] {$dA$} -- (h2);
 \node at (P) [above left] {$\hat{N}$};
\end{scope}

\begin{scope}[shift={(8.9,6.2)}, rotate=-30,scale=0.7]
 \coordinate (a)  at ( 0.00, 0.2);
 \coordinate (am) at (-0.05, 0.1);
 \coordinate (an) at ( 0.45, 0.8);

 \coordinate (b)  at ( 1.00, 1.0);
 \coordinate (bm) at ( 0.95, 0.9);

 \coordinate (c)  at ( 0.00, 2.0);
 \coordinate (cm) at (-0.05, 1.9);

 \coordinate (d)  at (-1.10, 1.3);
 \coordinate (dm) at (-1.15, 1.2);

 \draw[thin_line] (a) to[out= 25, in=200] (b);    
 \draw[thin_line] (b) to[out=115, in=-30] (c);    
 \draw[thin_line] (c) to[out=185, in= 60] (d);    
 \draw[thin_line] (d) to[out=-25, in=110] (a);    

 \draw[fill=gray!30] (a) to[out= 25, in=200] 
                     (b) to[out=115, in=-30] 
                     (c) to[out=185, in= 60] 
                     (d) to[out=-25, in=110] (a);    

 \draw[thin_line] (am) to[out= 25, in=200] (bm);    
 \draw[dash_line] (bm) to[out=115, in=-30] (cm);    
 \draw[dash_line] (cm) to[out=185, in= 60] (dm);    
 \draw[thin_line] (dm) to[out=-25, in=110] (am);    

 \draw[thin_line] (am) -- (a);
 \draw[thin_line] (bm) -- (b);
 \draw[dash_line] (cm) -- (c);
 \draw[thin_line] (dm) -- (d);

 \draw[fill=gray!90] (a)  to[out= 25, in=200] 
                     (b)  -- 
                     (bm) to[out=200, in=25]
                     (am) -- (a);

 \draw[fill=gray!90] (am) --
                     (a)  to[out=110, in=-25]
                     (d)  --
                     (dm) to[out=-25, in=110] (am);

 \coordinate (O)   at ( 0.00, 1.25);
 \coordinate (O1)  at ( 0.16, 1.10);
 \coordinate (O2)  at (-0.10, 1.18);
 \coordinate (P)   at ( 0.80, 2.10);
 \coordinate (h1)  at (-0.80, 2.00);
 \coordinate (h2)  at (-0.30, 1.60);

 \draw[vect_line] (O) -- (P);
 \draw[thin_line] (h1) node[above left] {$da$} -- (h2);
 \node at (P) [above=3pt, right=5pt] {$\hat{n}$};
\end{scope}

\tkzText(2.9 ,3.9){$B$}
\tkzText(8.7,4){$f(B)$}

\draw[line width=1pt, ->] (4.8,5) -- (7.6,5);
\tkzText(6.4,5.3){$f$}
 
\etkz
\end{center}
\caption{Illustration of the geometric setting of the Euclidean Piola identity.}
\label{fig:peppe}
\end{figure}

%

This paper is concerned with a generalization of the Piola identity to mappings between Riemannian manifolds. 
Let  $(\MOne,\gOne)$ and $(\MTwo,\gTwo)$ be smooth, oriented $d$-dimensional Riemannian manifolds. Then, for every $f\in C^2(\MOne;\MTwo)$, the \emph{Riemannian Piola identity} is
\beq
\label{eq:Piola_id_general_strong}
\delta_{\nabla^{f^*T\MTwo}} \Cof df =0,
\eeq
where the cofactor of $df$ is defined intrinsically (see \secref{sec:det_cof} for details) and $\delta_{\nabla^{f^*T\MTwo}}$ is the co-differential induced by the Riemannian connection on $f^*T\MTwo$ (see \secref{sec:coderivative}).
Equivalently, for every compactly-supported $\xi \in\Gamma(f^*T\MTwo)$,
\beq
\int_{\MOne} \innerp{\Cof df, \nabla^{f^*T\MTwo} \xi}_{\gOne,\gTwo}\,\dVolOne = 0,
\label{eq:Cof_d}
\eeq
where $\IP{\cdot}{\cdot}_{\gOne,\gTwo}$ is the inner-product on $T^*\MOne\otimes f^*T\MTwo$ induced by the metrics $\gOne$ and $\gTwo$.


The Piola identity for mappings between Riemannian manifolds was considered by Marsden and Hughes \cite[pp. 116--117]{MH83}. Let $f:\MOne \to \MTwo$ be a local diffeomorphism.
For  a vector field $X$ on $\MTwo$, its  \emph{Piola transform} \cite[Def. 7.18]{MH83}, $\piola(X)$,  is a vector field on $\MOne$,  \beq
\label{eq:Marsden_eq_3_i}
\piola(X) = (\Cof df)^T (f^*X),
\eeq
where $f^*X$ is the pullback of $X$. 
The identity derived in \cite{MH83} (where it is termed the Piola identity) is
\beq
\label{eq:Piola_Marsden_i}
\div \piola(X)= \brk{(\div X) \circ f}\, \Det df,
\eeq
which is a relation between the divergence of $X$ and the divergence of its Piola transform (here and below we denote by $\Det$ the (intrinsic) determinant of a linear map and by $\det$ the determinant of the matrix representing a map). 
We shall see below that
\beq
\label{eq:Marsden2i}
\div \piola(X) =  \brk{(\div X)\circ f} \, Det df- \IP{  f^*X}{\delta_{\nabla^{f^*T\MTwo}}\Cof df} ,
\eeq
which together with \eqref{eq:Piola_Marsden_i} implies \eqref{eq:Piola_id_general_strong}.

Marsden and Hughes further present a coordinate expression for the Riemannian Piola identity (i.e., a differential relation for $f$, which does not involve its action on vector fields). Their identity is however wrong; a derivation of the coordinate expression is presented in Appendix~\ref{app:coordinates}, and compared to the expression derived in \cite{MH83}.

Perhaps surprisingly, we show that the coordinate expression can be reduced into a form which does not involve the metrics $\g,\h$ at all, and looks identical to the Euclidean representation! After a second thought, this is not that surprising, as the coordinate representation of $f$ must satisfy an identity whose sole origin is the commutation of second derivatives.

The goal of this paper is to clarify several aspects of both Euclidean and Riemannian Piola identities, \eqref{eq:div_cof_grad_classic} and \eqref{eq:Piola_id_general_strong}. We prove the Riemannian Piola identity using two different approaches, based on two  different characterizations of the cofactor of a linear map.  

The first proof follows the continuum-mechanical approach depicted above, which is used to  derive Property \eqref{eq:Piola_Marsden_i} of the Piola transform. It relies on a characterization of the cofactor of a linear map via its action on $(d-1)$-dimensional surface elements. As a side result, our proof sheds new light on the classical proof in the Euclidean setting, showing that Eq.~\eqref{eq:peppe1}, which seems to be at the heart of the proof, is totally immaterial to that proof. 

The second proof was presented in \cite{KMS17}. It is based on a characterization of the cofactor as the derivative of the determinant. This paper contains a simplified Euclidean version of this proof, which motivates the steps in the more general setting of Riemannian manifolds.
We show that \eqref{eq:Piola_id_general_strong} is the Euler-Lagrange equation of a null-Lagrangian (a functional for which every map is critical), hence holds for every sufficiently regular map. A connection between the Piola identity and null-Lagrangians was already established in Evans \cite{Eva98}, where the existence of a null-Lagrangian was inferred from the Piola identity. We advocate that it should be viewed the other way around: the existence of a null-Lagrangian is the origin of the Piola identity. 
In fact, the Riemannian Piola identity is immediate once the null-Lagrangian has been identified, whereas its explicit  derivation is quite tedious. In both proofs, we start by reviewing the Euclidean case before proceeding to the Riemannian case.

This paper is structured as follows: 
In \secref{sec:prelims}, we introduce the geometric entities which play a part in the Riemannian Piola identity.   \secref{sec:elastic_approach} presents a proof based on the Piola transform. In \secref{sec:Null-Lagrangians_approach},  we present a proof  based on null-Lagrangians. In Appendix~\ref{app:deriv_det_bundle} we prove a lemma generalizing to the setting of vector bundles the well-known fact that the cofactor of a linear operator is the derivative of its determinant. Finally, we write in Appendix~\ref{app:coordinates} the Riemannian Piola identity in coordinates.

\section{Geometric preliminaries}
\label{sec:prelims}

\subsection{Determinant and cofactor of linear maps}
\label{sec:det_cof}

The notions of determinant and  cofactor of a linear map are at the heart of the Piola identity and its proof. While these notions are widely used in the context of a matrix representing a transformation with respect to orthonormal bases, it is valuable to present their intrinsic, coordinate-free definitions. For a $d$-dimensional oriented inner-product space $V$, we denote by  $\star_V^k:\Lambda_k(V)\to\Lambda_{d-k}(V)$ the Hodge-dual operators and by $\Vol_V$ the unit volume form (i.e., $\Vol_V(e_1,\dots,e_d)=1$ for every positively-oriented orthonormal basis of $V$).

\begin{definition}[determinant]
\label{def:intrinsic_det}
Let $V$ and $W$ be oriented, $d$-dimensional inner-product spaces. Let $A\in\Hom(V,W)$.
The determinant of $A$, $\Det A\in \Hom(\R,\R) \simeq \R$, is defined by
\[
\Det A := \star^d_W \circ \extp^d A \circ \star^0_V,
\]
where $\extp^d A = A \wedge \ldots \wedge A$, $d$ times, and we identify $\extp^0 V \simeq \extp^0 W \simeq \R$.
\end{definition}

If $(v_1,\dots,v_d)$ and $(w_1,\dots,w_d)$ are positively-oriented orthonormal bases for $V$ and $W$ respectively, and if $\hat{A}$ is the matrix representing $A$ with respect to these bases, then
\beq
\label{eq:det_is_action_cubes}
\extp^{d}  A (v_1 \wedge \dots \wedge v_d)= \det\hat{A} \, w_1 \wedge \dots \wedge w_d,
\eeq
from which follows that $\Det A = \det \hat{A}$. That is, the definition of the determinant of a linear map is consistent with the definition of the determinant of the matrix representing it with respect to any pair of positively-oriented orthonormal bases.

As is well-known, the determinant of a linear operator satisfies the following:

\begin{proposition}
\label{pn:Det_and_volumes_1}
Let $V$ and $W$ be oriented, $d$-dimensional inner-product spaces. 
Let $A\in\Hom(V,W)$.  Then,
\[
\Det A = \frac{A^* \Vol_W}{\Vol_V}.
\]
\end{proposition}

The proof is immediate from the definition, and can be obtained by choosing oriented orthonormal bases for $V$ and $W$. 

\begin{definition}[cofactor]
\label{def:intrinsic_cof}
Let $V$ and $W$ be oriented, $d$-dimensional inner-product spaces.
Let $A\in\Hom(V,W)$.
The cofactor  of $A$, $\Cof A \in\Hom(V,W)$, is defined by
\[
\Cof A := (-1)^{d-1} \star_W^{d-1} \circ \extp^{d-1} A \circ \star_V^1,
\]
where we identify $\extp^1 V \simeq V$ and $\extp^1 W \simeq W$.
\end{definition}

The intrinsic definition of the cofactor is consistent with the definition of the matrix-cofactor. The matrix representing $\Cof A$ is the matrix-cofactor of the matrix representing $A$, when the bases are positively-oriented and orthonormal.

While the determinant of a linear map encodes information about the action of that map on $d$-dimensional volume elements, the cofactor encodes information about the action of that map on $(d-1)$-dimensional hyper-cubes. Since the Hodge-dual operators are isometric isomorphisms, $\Cof A$ is essentially $\extp^{d-1} A$. For example, in the isotropic case, where  $V=W$  and $A=\lambda \id_V$, we obtain $\Cof A=\lambda^{d-1} \id_V$.

The cofactor and the determinant of a linear operator satisfy several relations which will be used throughout this paper. First, 
\beq
\label{eq:cof_det_expansion}
\Det A \, \id_V = A^T \circ \Cof A = (\Cof A)^T \circ A
\eeq
is an intrinsic version of a well-known property of the matrix cofactor; it is essentially the Laplace expansion of the determinant. 

The next proposition provides another relation between the cofactor and the determinant of  a linear map. 
Before stating it, we recall that in a $d$-dimensional oriented inner-product space, every unit vector $v\in V$ induces an orientation on its $(d-1)$-dimensional orthogonal complement, $\{v\}^\perp$: an orthonormal basis $(v_1,\dots,v_{d-1})$ for $\{v\}^\perp$ is positively-oriented if $(v,v_1,\dots,v_{d-1})$ is a positively-oriented orthonormal basis for $V$.

\begin{proposition}
\label{lem:Cof_encodes_restricted_det}
Let $\til V,\til W$ be oriented, $d$-dimensional inner-product spaces. Let $v^{\perp} \in \til V$ and $w^{\perp} \in \til W$ be unit vectors, and denote by $V=\{v^{\perp}\}^\perp \subseteq \til V$ and $W=\{w^{\perp}\}^\perp  \subseteq \til W$ their $(d-1)$-dimensional orthogonal complements, with the orientations induced by $(\til V,v^\perp)$ and $(\til W,w^\perp)$. 
Let $\til A\in\Hom(\til V, \til W)$ satisfy $\til A(V) \subseteq W$; denote $A=\til A|_V\in\Hom(V,W)$.  Then,
\[
\Cof \til A (v^\perp)= \Det A\, w^\perp.
\]
\end{proposition}

\begin{proof}
Let $(v_1,\dots,v_{d-1})$ and $(w_1,\dots,w_{d-1})$ be positively-oriented orthonormal bases for $V$ and $W$, respectively. Then,
\[
\begin{split}
\Cof \til A (v^\perp) &=(-1)^{d-1}  \star^{d-1}_{\til W}  \extp^{d-1} \til A \star_{\til V}^1 (v^\perp) \\
&=(-1)^{d-1}  \star^{d-1}_{\til W}  \extp^{d-1} \til A (v_1 \wedge \dots \wedge v_{d-1}) \\
&= (-1)^{d-1}  \star^{d-1}_{\til W}  \extp^{d-1}  A (v_1 \wedge \dots \wedge v_{d-1}) \\
&=(-1)^{d-1}  \star^{d-1}_{\til W}  \Det A\, (w_1 \wedge \dots \wedge w_{d-1}) \\
&=\Det A  \, w^\perp,
\end{split}
\]
where in the passage to the third line we used the fact that $A = \til A|_V$, and the passage to the fourth line follows from the intrinsic definition of the determinant. The first and the last equalities follow from the fact that $(v^\perp,v_1,\dots,v_{d-1})$ and $(w^\perp,w_1,\dots,w_{d-1})$ are positively-oriented orthonormal bases for $\til V$ and $\til W$.
\end{proof}

\subsection{Determinant and cofactor of linear bundle maps}
\label{sec:det_cof2}

In this section we apply the linear-algebraic constructs of the previous section to linear bundle maps between vector bundles. 
Let  $(\MOne,\gOne)$ and $(\MTwo,\gTwo)$ be smooth, oriented $d$-dimensional Riemannian manifolds.
We denote by $\star_{\MOne}^k:\Lambda_k(T\MOne)\to\Lambda_{d-k}(T\MOne)$ and $\star_{\MTwo}^k:\Lambda_k(T\MTwo)\to\Lambda_{d-k}(T\MTwo)$ the Hodge-dual operators of the tangent bundles (note that the Hodge-dual in Riemannian settings usually acts on the exterior algebra of the cotangent bundle). We denote by $\dVolOne$ and $\dVolTwo$ the corresponding volume forms. 

Let $f:\MOne\to \MTwo$ be a differentiable mapping; its differential is a linear bundle map,
\[
df: T\MOne \to f^* T\MTwo.
\]
The determinant of $df$,
\[
\Det df = \star^d_{\MTwo} \circ \extp^d df \circ \star^0_{\MOne}
\]
is a function on $\MOne$, whereas its cofactor,
\[
\Cof df = (-1)^{d-1} \star^{d-1}_{\MTwo} \circ \extp^{d-1} df \circ \star^1_{\MOne}
\]  
is a section of $T^*\MOne\otimes f^*T\MTwo$.
By \propref{pn:Det_and_volumes_1}, the determinant is the ratio of the volume forms,
\[
\Det df = \frac{f^\star \dVolTwo}{\dVolOne}.
\]

Let $\M$ be a smooth, oriented $d$-dimensional manifold with boundary and let 
$p \in \partial\M$.  A vector $v \in T_p\M \setminus T_p\partial\M$ is called  \emph{outward-pointing} if for some $\ep >0$ there exists a smooth curve $\gamma:(-\ep,0] \to \M$ such that $\ga(0)=p$ and $\dot \ga(0)=v$.
Let $\xi$ be an outward-pointing vector field on $\partial\M$;  $\xi$ induces an orientation on $T\partial\M$, called the Stokes orientation: for $p\in\partial\M$, $(v_1,\dots,v_{d-1})$ is a positively-oriented basis for $T_p\partial\M$ if $(\xi_p,v_1,\dots,v_{d-1})$ is a positively-oriented basis for $T_p\M$. This orientation does not depend on the choice of the outward-pointing vector field $\xi$. The Stokes  orientation is naturally diffeomorphic-invariant in the following sense:

\begin{lemma}
\label{lem:naturaliy_Stokes_orientation}
Let $f:\MOne \to \MTwo$ be an orientation-preserving diffeomorphism between $d$-dimensional oriented manifolds with boundaries. Then 
\[
f|_{\pl \MOne}:\pl \MOne \to \pl \MTwo
\] 
is also an orientation-preserving diffeomorphism, where the orientations on the boundaries are the induced Stokes orientations.
\end{lemma}

\begin{proof}
Let $\xi$ be some outward-pointing vector field on $\pl \MOne$.  Then $df(\xi)$ is outward-pointing on $\pl \MTwo$.
Let $p \in \partial\MOne$ and suppose that $(v_1,\dots,v_{d-1})$ is a positive basis for $T_p \partial\MOne$. By definition, $(\xi_p,v_1,\dots,v_{d-1}) $ is a positive basis for $T_p\MOne$. Since $f$ is orientation preserving,  $(df_p(\xi_p),df_p(v_1),\dots,df_p(v_{d-1}))$ is a positive basis for $T_{f(p)}\MTwo$. Since $df_p(\xi_p)$ is outward-pointing in $T_{f(p)}\MTwo$, this implies $(df_p(v_1),\dots,df_p(v_{d-1}))$ is a positive basis for $T_{f(p)} \partial\MTwo$. Note that for every $i=1,\dots,d-1$, 
\[
df_p(v_i)=df_p|_{T_p \partial\MOne}(v_i)=d(f|_{\partial\MOne})_p(v_i).
\]
\end{proof}

We may now apply \propref{lem:Cof_encodes_restricted_det} to maps between manifolds:

\begin{proposition}
\label{prop:det_cof_bundle}
Let $(\MOne,\gOne)$ and $(\MTwo,\gTwo)$ be smooth, oriented $d$-dimensional Riemannian manifolds with boundaries and let $f:\MOne\to\MTwo$ be a diffeomorphism. Let $\nu_1$ and $\nu_2$ be the unique outward-pointing normal unit vector fields on $\partial\M_1$ and $\partial\M_2$. Then, 
\[
\Cof df(\nu_1) = \Det(df|_{T\partial \MOne}) f^* \nu_2,
\]
which is an identity between sections of $f^* T\MTwo$.
\end{proposition}

\begin{proof}
This is an immediate application of \propref{lem:Cof_encodes_restricted_det}, with  $\til V = T_p \MOne$, $\til W = T_{f(p)}\MTwo$, $v^\perp=(\nu_1)_p$,  $w^\perp = (\nu_2)_{f(p)}$ and $\til A = df_p$ at every $p\in\partial\MOne$.
\end{proof}

The last relation between the cofactor and the determinant of a linear bundle map states that the cofactor is the derivative of the determinant: 

\begin{lemma}
\label{lem:Cofactor_grad_Determinant_bundle}
Let $E$ and $F$ be oriented vector bundles of rank $d$  over a smooth manifold $\M$, equipped with smooth metrics and metrically-compatible connections. Let $A:E \to F$ be a smooth bundle map. Then, for every $X \in \Gamma(T\MOne)$
\[
d(\Det A)(X)= \IP{\Cof A}{\nabla_X A}_{E,F},
\]
where $\Det A$ and $\Cof A$ are defined as in \ref{def:intrinsic_det} and \ref{def:intrinsic_cof}, using the metrics and orientations on $E,F$, and $\nabla$ is the tensor-product connection on $E^* \otimes F$ induced by the connections on $E$ and $F$.
 \end{lemma}

The proof is given in Appendix~\ref{app:deriv_det_bundle}.

\subsection{The coderivative for vector-valued forms}
\label{sec:coderivative}


Let $(\M,\g)$ be a $d$-dimensional oriented Riemannian manifold.
Let $E$ be a vector bundle over $\M$ (of arbitrary rank $n$), endowed with a Riemannian metric $\h$, and a metrically-consistent affine connection $\nabla^E$. That is, for every pair of sections $\xi,\eta\in\Gamma(E)$, and every vector field $X\in\Gamma(T\M)$,
\[
X(\h(\xi,\eta)) = \h(\nabla^E_X\xi,\eta) + \h(\xi,\nabla^E_X\eta).
\]

We denote by $\Omega^1(\M;E) = \Gamma(T^*\M\otimes E)$ the space of $1$-forms on $\M$ with values in $E$.
For a section $\xi$ of $E$, its covariant derivative 
\[
\nabla^E\xi : X \mapsto \nabla_X^E\xi
\]
is an element of $\Omega^1(\M;E)$.
Finally, the metrics on $\M$ and $E$ induce a metric on $\Omega^1(\M;E)$, denoted $\innerp{\cdot,\cdot}_{\g,\h}$. 
With that, we recall the definition of the coderivative for vector-valued forms:

\begin{definition}
\label{def:coderivative}
The coderivative,
\[
\delta_{\nabla^E}: \Omega^1(\M;E) \to  \Omega^{0}(\M;E) \simeq \Gamma(E)
\]
is the adjoint of the connection $\nabla^E:\Gamma(E)\to \Omega^1(\M;E)$ with respect to the metric $\innerp{\cdot,\cdot}_{\g,\h}$.
That is, 
\[
\int_{\MOne} \innerp{\sigma,\delta_{\nabla^E} \rho}_{\g,\h}\,d\Volume = \int_{\MOne} \innerp{\nabla^E \sigma,\rho}_{\g,\h}\,d\Volume,
\]
for all $\rho\in\Omega^1(\M;E)$ and compactly-supported $\sigma\in \Gamma(E)$.
\end{definition}

The coderivative of a vector-valued form has a well-known explicit formula.
Let $\omega \in \Omega^1(\M;E)$. Given an orthonormal frame $E_i$ for $T\M$, $\delta_{\nabla^E}$ is given by
\beq
\delta_{\nabla^E}\omega  =-\sum_{i=1}^d (\nabla_{E_i} \omega)(E_i)=-\trg(\nabla \omega),
\label{eq:codiv_formula1}
\eeq
where $\nabla \omega$ is the connection induced  on $T^*\M \otimes E$ by the Levi-Civita connection on $\M$ and $\nabla^E$ (see e.g. \cite[Lemma 1.20]{EL83} for a proof). 

We will use the coderivative in the following setting: 
Let $f:\MOne\to\MTwo$ be smooth. Its differential $df$ can be viewed as a vector-valued form 
\[
df\in \Gamma(T^*\MOne\otimes f^*T\MTwo)\simeq \Omega^1(\MOne;f^*T\MTwo). 
\]
Then, $\Cof df  \in \Omega^1(\MOne;f^*T\MTwo)$ is of the same type as $df$. Hence, $\delta_{\nabla^{f^*T\MTwo}}  \Cof df $ is well-defined according to \defref{def:coderivative}, with $E=f^*T\MTwo$.

\section{The Piola transform approach}
\label{sec:elastic_approach}

In this section we present a proof of the Riemannian Piola identity \eqref{eq:Piola_id_general_strong} in the spirit of the continuum-mechanics approach briefly reviewed in the Introduction. It is based on the property of the cofactor established in \propref{prop:det_cof_bundle}. The crux of the proof is pulling back integrals from the target manifold $\MTwo$ to the source manifold $\MOne$ via the map $f$. 

A limitation of this derivation is that it only works for local diffeomorphisms. 
This is however not a severe limitation: every second-order equation satisfied by all local diffeomorphisms extends automatically to smooth maps. This follows from the facts that a $k$th-order differential equation for $f$ can be viewed as an algebraic equation for its $k$-jet \cite{Sau89}, and that jets of invertible maps are dense in the space of jets.

\subsection{The Euclidean case}

\begin{proposition}
\label{prop:Cof_orthog_Euc}
Let $\Omega_1, \Omega_2 \subseteq \R^d$ be open and bounded domains with a $C^2$-boundary, and let  $f:\bar \Omega_1 \to \bar \Omega_2$ be a $C^2$-diffeomorphism. 
For every $1 \le i \le d$,
\[
\int_{\Omega_1}  \div \brk{ \brk{\cof \nabla f}^T e_i } dV=0,
\]
where $(\cof \nabla f)^T e_i$ is the $i$-th row of the cofactor matrix $\cof \nabla f$ and $dV$ is the standard volume form in $\R^d$.
\end{proposition}

\begin{proof}
Without loss of generality, we may assume that $f$ is orientation-preserving.
We denote by $dS_1$ and $dS_2$ the surface forms of $\pl \Omega_1$ and $\pl \Omega_2$. For every $1\le i\le d$,
\beq
\label{eq:main_elastic_i_Euc}
\begin{split}
0&=\int_{\Omega_2} \div e_i \, dV \\
&= \int_{\pl \Omega_2}\IP{e_i}{\nu_2} \,dS_2 \\
&= \int_{\pl \Omega_1} \IP{e_i}{\nu_2 \circ f} \,f^*dS_2 \\
&= \int_{\pl \Omega_1} \IP{e_i }{\Det \brk{df|_{T\pl \Omega_1 }}  (\nu_2 \circ f)} \,dS_1 \\
&= \int_{\pl \Omega_1} \IP{e_i }{\cof \nabla f (\nu_1) } \,dS_1 \\
&= \int_{\pl \Omega_1} \IP{(\cof \nabla f)^T (e_i )}{ \nu_1 } \,dS_1  \\
&= \int_{\Omega_1}  \div \brk{ \brk{\cof \nabla f}^T e_i } \,dV.
\end{split}
\eeq
The first equality holds because $e_i$ is divergence-free. 
The passage to the second line follows from the divergence theorem. The passage to the third line is obtained by a pullback (change of variables). The passage to the fourth line results from the relation 
$f^*dS_2 = \Det (df|_{T\pl \Omega_1})\, dS_1$.
The passage to the fifth line follows from \propref{prop:det_cof_bundle}. Finally, the last equality is obtained by another application of the divergence theorem. 
\end{proof}

\begin{corollary}
\label{cor:euc_piola}
With the same notation as above, let $f:\bar \Omega_1 \to \bar \Omega_2$ be a $C^2$ local diffeomorphism.  Then, the Euclidean Piola identity \eqref{eq:div_cof_grad_classic} holds.
\end{corollary}

\begin{proof}
Since the claim is local, we may assume that $f$ is a diffeomorphism. Let $p \in \Omega_1$,  
and apply \propref{prop:Cof_orthog_Euc} on shrinking balls around $p$. The integrand vanishes by a standard differentiation argument.
\end{proof}

\subsection{The Riemannian case}
\label{sec:mechanics_Riemannian}

It is not immediately obvious how to generalize \propref{prop:Cof_orthog_Euc} to maps between Riemannian manifolds.
The first equality in \eqref{eq:main_elastic_i_Euc} is due to the ``constantness" of the frame field $\{e_i\}$, which in turn implies the vanishing of the surface integral of the unit normal to the boundary.  On a non-Euclidean manifold, there is no such thing as a parallel frame-field, or any canonical notion of a divergence-free frame field. 

It turns out, however, that the existence of a parallel frame field is not really needed for the  proof of \propref{prop:Cof_orthog_Euc}. 
A more careful examination of the proof reveals the occurrence of a cancellation, which as a result of, the particular properties of the frame field $\{e_i\}$ are immaterial. We prove the Riemannian Piola identity in two steps, the first of which was proved in \cite{MH83}:

\begin{proposition}[Marsden and Hughes]
\label{prop:MH83}
Let $f:\MOne \to \MTwo$ be a local diffeomorphism between oriented, compact, $d$-dimensional Riemannian manifolds with boundaries.  
For any vector field $X\in\Gamma(T\MTwo)$,
\[
\div \piola(X) =  \brk{(\div X)\circ f} \, \Det df,
\]
where $\piola(X)\in \Gamma(T\MOne)$ is defined by \eqref{eq:Marsden_eq_3_i}.
\end{proposition}

\begin{proof}
Since the claim is local, we may assume that $f$ is a diffeomorphism. 

Without loss of generality, we may assume that $f$ is orientation-preserving (otherwise, reverse the orientation of, say, $\MOne$).
By \lemref{lem:naturaliy_Stokes_orientation}, $f|_{\pl \MOne}:\pl \MOne \to \pl \MTwo$ is an orientation-preserving (with respect to the induced Stokes orientations) diffeomorphism. 
Denote by $dS_1$ and $dS_2$ the surface volume forms on $\MOne$ and $\MTwo$ induced by $\dVolOne$ and $\dVolTwo$.
Note that 
\beq
(f|_{\pl \MOne})^*dS_2=\Det \brk{d\brk{f|_{\pl \MOne}}}dS_1=\Det \brk{\brk{df}|_{T\partial\MOne}}dS_1.
\label{eq:det_bdry}
\eeq
Replicating the steps in \eqref{eq:main_elastic_i_Euc} 
\beq
\label{eq:main_elatic}
\begin{split}
\int_{\MTwo} \div X \, \dVolTwo &= \int_{\pl \MTwo}\IP{X}{\nu_2} \, dS_2\\
&=\int_{\pl \MOne} \IP{f^*X}{f^*\nu_2} \, f^*dS_2 \\
&=\int_{\pl \MOne} \IP{f^*X}{\Det \brk{df|_{T\partial\MOne}}  (f^*\nu_2)} dS_1 \\
&=\int_{\pl \MOne} \IP{f^*X}{\Cof df (\nu_1) } \, dS_1\\
&=\int_{\pl \MOne} \IP{\piola(X)}{ \nu_1 } \,dS_1 \\
&=\int_{\MOne} \div \piola(X) \,\dVolOne.
\end{split}
\eeq
where in the passage to the  third line we used \eqref{eq:det_bdry}, in the passage to the fourth line we used \propref{prop:det_cof_bundle}, in the passage to the fifth line  we used the definitions of the transpose,  and in the passage to the last line we used once again the divergence theorem. Pulling back the left-hand side, we obtain that
\beq
\label{eq:orthogonality_cof_a_1}
\int_{\MOne} ((\div X) \circ f) \, \Det df\,\dVolOne = \int_{\MOne} \div \piola(X) \,\dVolOne,
\eeq
Since this identity holds for $\MOne$ replaced by any open subset, the integrands are equal, which completes the proof.
\end{proof}

\begin{proposition}
\label{prop:new_prop}
Let $f:\MOne \to \MTwo$ be a smooth map between oriented, compact, $d$-dimensional Riemannian manifolds with boundaries.  
For any vector field $X\in\Gamma(T\MTwo)$,
\beq
\label{eq:Marsden2i}
\div \piola(X)= ((\div X)\circ f)\, \Det df- \IP{  f^*X}{\delta_{\nabla^{f^*T\MTwo}}\Cof df}.
\eeq
\end{proposition}

\begin{proof}
Let $e_i$ be an orthonormal frame for $T\MOne$. Writing the expression for the divergence with respect to an orthonormal frame,  and applying the Leibniz rule,
\beq
\label{eq:orthogonality_cof_a_2}
\begin{split}
\div \piola(X) 
&=\sum_{i=1}^d \IP{\nabla_{e_i}^{T\MOne} \brk{(\Cof df)^T (f^*X)}}{e_i} \\ 
&=\sum_{i=1}^d \IP{ \brk{\nabla_{e_i} \brk{\Cof df}^T} (f^*X)+ (\Cof df)^T \brk{\nabla_{e_i}^{f^*T\MTwo}(f^*X)}}{e_i} .
\end{split}
\eeq
Examining the first summand, 
\beq
\label{eq:orthogonality_cof_a_3}
\begin{split}
\sum_{i=1}^d \IP{ \brk{\nabla_{e_i} \brk{\Cof df}^T} (f^*X )}{e_i}&=\sum_{i=1}^d \IP{ \brk{\nabla_{e_i} \Cof df}^T (f^*X)}{e_i} \\
&=\sum_{i=1}^d \IP{  f^*X}{\brk{\nabla_{e_i} \Cof df} (e_i)}\\ 
&=- \IP{  f^*X}{\delta_{\nabla^{f^*T\MTwo}}\Cof df} ,
\end{split}
\eeq
where we used the fact that the transpose operator commutes with covariant derivative, and the expression \eqref{eq:codiv_formula1} for $\delta_{\nabla^{f^*T\MTwo}}$.

Examining the second summand in \eqref{eq:orthogonality_cof_a_2}, 
\beq
\label{eq:orthogonality_cof_a_4}
\begin{split}
\sum_{i=1}^d \IP{ (\Cof df)^T \brk{\nabla_{e_i}^{f^*T\MTwo}(f^*X)}}{e_i} &=\sum_{i=1}^d \IP{(\Cof df)^T f^*\brk{\nabla_{f_*(e_i)}^{T\MTwo}X }}{e_i} \\
&=\Det df \, \sum_{i=1}^d \IP{(df)^{-1} f^*\brk{\nabla_{f_*(e_i)}^{T\MTwo}X }}{e_i} \\
&= \Det df \, \sum_{i=1}^d \IP{(df)^{-1} \circ f^*\brk{\nabla^{T\MTwo}X } \circ df (e_i)}{e_i} \\
&= \Det df \, \tr \brk{ (df)^{-1} \circ f^*\brk{\nabla^{T\MTwo}X } \circ df } \\ 
&= \Det df \, \tr \brk{  f^*\brk{\nabla^{T\MTwo}X } } \\
&= ((\div X)\circ f)\, \Det df ,
\end{split} 
\eeq
where in the passage to the second line we used \eqref{eq:cof_det_expansion}, and in the passage to the fifth line we used the cyclic invariance of the trace: $\tr(T^{-1}ST)=\tr(S)$.

Equations \eqref{eq:orthogonality_cof_a_2}, \eqref{eq:orthogonality_cof_a_3} and \eqref{eq:orthogonality_cof_a_4} yield the desired result. 
\end{proof}

\begin{corollary}
Let $f:\MOne \to \MTwo$ be a local diffeomorphism between smooth, oriented, $d$-dimensional Riemannian manifolds with boundaries. 
Then the Riemannian Piola identity \eqref{eq:Piola_id_general_strong} holds.
\end{corollary}

\begin{proof}
Combining \propref{prop:MH83} and \propref{prop:new_prop}, if follows that for every vector field $X\in\Gamma(T\MTwo)$,
\[
\IP{  f^*X}{\delta_{\nabla^{f^*T\MTwo}}\Cof df} = 0.
\]
Taking $f^*X= \delta_{\nabla^{f^*T\MTwo}}\Cof df$, the result follows.
\end{proof}

\section{Null-Lagrangian approach}
\label{sec:Null-Lagrangians_approach}

Let $\calX$ be a function space of smooth maps between two domains.   
A functional $E:\calX\to\R$ is a \emph{null-Lagrangian} if every smooth map is a critical point of $E$, with respect to smooth homotopies relative to the boundary. As a result, every function in $\calX$ satisfies the Euler-Lagrange equation corresponding to $E$.
In this section we prove the Piola identity by showing that it is the Euler-Lagrange equation of a null-Lagrangian. As in the previous section, we start by considering the Euclidean setting, and then generalize the treatment to mappings between Riemannian manifolds.  

\subsection{The Euclidean case}
\label{sec:Null_Euc}

We begin by following the treatment of Iwaniec \cite{Iwa14}.
\begin{lemma}[Iwaniec]
\label{lem:tech_Jac1}
Let $\Omega \subseteq \R^d$ be open and bounded, and let $f:\Omega \to \R^d$ be smooth. For every $i=1,\dots,d$,
\[
\label{lem:det_wedge_express}
\begin{split}
\det \nabla f \,dV&= \det \nabla f \, dx^1 \wedge dx^2 \dots \wedge dx^d \\
&= df^1 \wedge df^2 \dots \wedge df^d \\
&= (-1)^{i-1}\, d \big(df^{1} \wedge df^{2} \wedge \dots \wedge df^{i-1} \wedge  f^i  \wedge df^{i+1} \wedge \dots \wedge df^d \big).
\end{split}
\]
\end{lemma}

\begin{proof}
The first equality is just a rewriting of the unit volume element as a wedge product of the standard co-frame.
The last equality follows directly from the Leibniz rule for the exterior derivative and the fact $d^2=0$. As for the middle equality,
\[
\begin{split}
df^1 \wedge df^2 \dots \wedge df^d 
&=\pd{f^1}{x_{j_1}} \pd{f^2}{x_{j_2}} \dots \pd{f^d}{x_{j_d}}  \,   dx^{j_1}\wedge dx^{j_2} \dots \wedge dx^{j_d}\\
&= \sum_{\sigma \in S^d}\pd{f^1}{x_{\sigma(1)}} \pd{f^2}{x_{\sigma(2)}} \dots \pd{f^d}{x_{\sigma(d)}}  \,  dx^{\sigma(1)}\wedge dx^{\sigma(2)} \dots \wedge dx^{\sigma(d)}\\
&= \sum_{\sigma \in S^d} \pd{f^1}{x_{\sigma(1)}} \pd{f^2}{x_{\sigma(2)}} \dots \pd{f^d}{x_{\sigma(d)}} \sgn(\sigma)  \, dx^1 \wedge dx^2 \dots \wedge dx^d \\
&= \det \nabla f \, dx^1 \wedge dx^2 \dots \wedge dx^d.
\end{split}
\]
\end{proof}

\begin{corollary}
\label{cor:det_depend_only_bdry}
The integral $\int_{\Omega} \det \nabla f\, dV$ only depends on the value of $f$ on $\partial \Omega$.
\end{corollary}

\begin{proof}
Suppose that $f|_{\partial \Omega}=g|_{\partial \Omega}$.
Using a telescopic sum, 
\[
\begin{split}
(\det \nabla f- \det  \nabla g)\, dV &= df^1 \wedge df^2 \dots \wedge df^n-dg^1 \wedge dg^2 \dots \wedge dg^d \\
&= \sum_{i=1}^d dg^1 \wedge \dots \wedge dg^{i-1} \wedge d(f^i-g^i) \wedge df^{i+1} \wedge \dots \wedge df^d  \\
&= \sum_{i=1}^d  (-1)^{i-1} d\brk{dg^1 \wedge \dots \wedge dg^{i-1} \wedge (f^i-g^i) \wedge df^{i+1} \wedge \dots \wedge df^d}.
\end{split}
\]
Hence,
\[
\begin{split}
\int_{\Omega} (\det  \nabla f - \det \nabla g)\, dV &= \sum_{i=1}^d (-1)^{i-1} \int_{\Omega} d\brk{ dg^1 \wedge \dots \wedge dg^{i-1} \wedge (f^i-g^i) \wedge df^{i+1} \wedge \dots \wedge df^d} \\
&=  \sum_{i=1}^d (-1)^{i-1} \int_{\partial \Omega} dg^1 \wedge \dots \wedge dg^{i-1} \wedge (f^i-g^i) \wedge df^{i+1} \wedge \dots \wedge df^d=0,
\end{split}
\]
where the last equality follows from the assumption that $f|_{\partial \Omega}=g|_{\partial \Omega}$.

\end{proof}

\corref{cor:det_depend_only_bdry} immediately implies the following:

\begin{corollary}
\label{cor:det_Null_Euc}
The functional 
\[
E(f)=\int_{\Omega} \det \nabla f\, dV
\]
is a null-Lagrangian.
\end{corollary}

\begin{proposition}
\label{lem:E-l_det_Euc}
The Euler-Lagrange equation of the functional
\beq
\label{eq:Jacobian_functional_Euc}
E(f)=\int_{\Omega} \det  \nabla f \, dV
\eeq
is $\div \cof \nabla f = 0$.
\end{proposition}

\begin{proof}
Let $f_t=f+tw$ for some smooth vector field $w:\R^d \to \R^d$. Then,
\[
\begin{split}
\left. \frac{d}{dt} \right|_{t=0} E(f_t) &= \int_{\Omega} \left. \frac{d}{dt} \right|_{t=0}  \det \nabla f_t \, dV \\
&=\int_{\Omega} \left. \frac{d}{dt} \right|_{t=0}  \det (\nabla f+t \, \nabla w) \, dV \\
&=\int_{\Omega} \langle \cof  \nabla f, \nabla w \rangle \, dV  \\
&=- \int_{\Omega} \langle \text{div} \cof \nabla f,w \rangle \, dV,
\end{split}
 \]
where we used the well-known facts that 
the cofactor is the gradient of the determinant and that the divergence is the (negative) adjoint of the gradient.
\end{proof}

\subsection{The Riemannian case}

We now turn to the Riemannian case. Let $\MOne$ and $\MTwo$ be smooth, oriented, $d$-dimensional  Riemannian manifolds, and set
\beq
\label{eq:Jacobian_functional}
E(f)=\int_{\MOne} f^* \dVolTwo=\int_{\MOne} \Det df\, \dVolOne.
\eeq

We start by observing that there is an obstacle in generalizing the approach used in \secref{sec:Null_Euc}:  \corref{cor:det_depend_only_bdry} does not hold 
for mappings between arbitrary Riemannian manifolds. 
Consider the following counter-example: 
Let $\MOne$ be a hemisphere and let $\MTwo$ be a sphere. Let $f_1,f_2:\MOne \to \MTwo$ be two embeddings of the hemisphere in the sphere coinciding on the boundary; $f_1$ maps $\MOne$ onto the upper hemisphere of $\MTwo$ and $f_2$ maps $\MOne$ onto the lower hemisphere of $\MTwo$. Endow $\MTwo$ with a Riemannian metric, such that the volumes of the two hemispheres are different. Then 
\[
\int_{\MOne} \Det df_1\,\dVolOne \neq \int_{\MOne} \Det df_2\,\dVolOne
\]
even though $f_1|_{\partial \MOne}  =f_2|_{\partial \MOne}$.
This shows that the "decomposition" technique used in \lemref{lem:tech_Jac1} to prove that the integral of the determinant is a 
null-Lagrangian does not work for arbitrary manifolds. The obstruction to such a generalization turns out to be of topological nature.

The statement that $E(f)$ depends only on $f|_{\pl \MOne}$ is strictly stronger than the statement that it is invariant under a homotopy relative to $\pl \MOne$. The relevant generalization of \corref{cor:det_depend_only_bdry} for mappings between manifolds is the following:

\begin{lemma}
\label{lem:integrals_exact_forms_depend_only_bd}
Let $\MOne$ and $\MTwo$ be smooth manifolds of dimensions $m_1$ and $m_2$, respectively (possibly with boundaries). Suppose that $\MOne$ is compact and oriented. Let $\omega \in \Omega^{m_1}(\MTwo)$ be an exact $m_1$-form on $\MTwo$. 
Let $f_0, f_1:\MOne\rightarrow \MTwo$ be smooth maps coinciding on $\partial \MOne$.
Then
 \[
\int_{\MOne}f_0^*\omega=\int_{\MOne} f_1^*\omega.
 \]
 \end{lemma}

\begin{comment}
Every closed form on $\R^d$ is exact, and in particular, the standard volume form $dV$. This gives an alternative proof for \corref{cor:det_depend_only_bdry}.
\end{comment}

\begin{proof}
Let $f:\MOne \to \MTwo$ be smooth.
By assumption, $\omega=d\eta$ for some $\eta \in \Omega^{m_1-1}(\MTwo)$. Using the commutation of exterior differentation and pullbacks,
\[
\int_{\MOne}f^*\omega=\int_{\MOne}f^*d\eta=\int_{\MOne}df^*\eta=\int_{\partial\MOne}f^*\eta,
\]
which only depends on $f|_{\partial\MOne}$.
\end{proof}

\lemref{lem:integrals_exact_forms_depend_only_bd} can be used to prove that $E$ is a null-Lagrangian as follows: Since being a null-Lagrangian is equivalent to the satisfaction of the Euler-Lagrange equation by every map, this a local property. That is, it can be checked for maps between small balls (or other manifolds diffeomorphic to $\R^d$).  Since $f^* \dVolTwo$ is closed,  it is locally exact. Thus, \lemref{lem:integrals_exact_forms_depend_only_bd} implies the null-Lagrangian property.

We provide here a different argument, which is also classical:

\begin{lemma}
\label{lem:integrals_homotopic}
Let $\MOne$ and $\MTwo$ be smooth manifolds of dimensions $m_1$ and $m_2$, respectively (possibly with boundaries). Suppose that $\MOne$ is compact and oriented.  Let $\omega \in \Omega^{m_1}(\MTwo)$ be a closed $m_1$-form on $\MTwo$. 
Let $f_0, f_1:\MOne\rightarrow \MTwo$ be smooth maps that are homotopic relative to $\partial \M$.
Then,
 \[
\int_{\MOne}f_0^*\omega=\int_{\MOne} f_1^*\omega.
 \]
\end{lemma}

\begin{proof}
Let $F:\MOne\times I\rightarrow \MTwo$ be a smooth homotopy between $f_0$ and $f_1$ relative to $\pl \MOne$, i.e., $f_t|_{\partial \MOne}=f_0|_{\partial \MOne}$ for every $t$. Then,
\[
\begin{split}
 0 &= \int_{\MOne\times I} F^\ast d\omega \\
 &= \int_{\MOne\times I} dF^\ast \omega \\
 &= \int_{\partial(\MOne\times I)} F^\ast \omega \\ 
 &= \int_{\MOne\times\{1\}} F^\ast \omega - \int_{\MOne\times\{0\}}F^\ast \omega + \int_{\partial \MOne \times (0,1)} F^\ast \omega \\ &= \int_{\MOne} f_1^\ast\omega - \int_{\MOne}f_0^\ast \omega,
\end{split}
\]
The first equality follows from $\omega$ being closed, i.e., $d\omega=0$. The passage to the second line follows from the commutation of exterior differentiation and pullbacks. The passage to the third line follows from Stokes' theorem.
The passage to the fourth line follows from the decomposition of $\partial(\MOne\times I)$ into three submanifolds.  Finally, in the passage to the fifth line we used the fact that $F|_{\MOne\times\{0\}} = f_0$, $F|_{\MOne\times\{1\}} = f_1$, and the fact that since $F$ is a homotopy relative to $\partial \MOne$, the restriction of $F^\ast\omega$ to $\partial \MOne \times (0,1)$ vanishes.
\end{proof}

\begin{corollary}[Pullbacks of closed forms are null-Lagrangians]
\label{cor:Null-Lag1}
Let $\MOne$, $\MTwo$ and $\omega$ be defined as in \lemref{lem:integrals_homotopic}. Let $E:C^{\infty}(\MOne,\MTwo) \to \R$ be given by \[
E(f)=\int_{\MOne} f^*\omega.  
\]
Then $E$ is a null-Lagrangian. 
\end{corollary}

\begin{proof}
Let $f_t:\MOne \to \MTwo$ be a smooth variation relative to $\pl \MOne$ of $f_0=f$. By \lemref{lem:integrals_homotopic}, 
$E(f_t)=E(f_0)$, hence $E(f_t)$ is constant. 
\end{proof}

In the case where $m_1=m_2$, every $m_1$-form on $\MTwo$ is closed. Hence the functional $E$ defined by \eqref{eq:Jacobian_functional} is a null-Lagrangian.

Note that  we limited our treatment to compact domains. For non-compact domains, one has to restrict the functional to compact subsets of $\MOne$ and consider compactly-supported variations. 

We proceed to derive the Euler-Lagrange equation for the functional \eqref{eq:Jacobian_functional}:

\begin{proposition}
\label{lem:E-l_det}
Let $\MOne$ and $\MTwo$ be smooth, oriented $d$-dimensional  Riemannian manifolds; The Euler-Lagrange equation for the functional
\beq
\label{eq:Jacobian_functional}
E(f)=\int_{\MOne} f^* \dVolTwo=\int_{\MOne} \Det df\, \dVolOne
\eeq
is $\delta_{\nabla^{f^*T\MTwo}} \Cof df = 0$
\end{proposition}

\begin{proof}
Let $\phi: \MOne \to \MTwo$ be a smooth map, and let $V \in \Ga(\phi^*T\MTwo)$. Let $\phi_t:\MOne \to \MTwo$ be a smooth variation constant on $\pl \MOne$, such that $\phi_0=\phi$ and $\left. \partial\phi_t/\partial t  \right|_{t=0}=V$. Our goal is to prove that
\[
\left. \deriv{}{t} E(\phi_t) \right|_{t=0}=\int_{\MOne} \IP{\delta_{\nabla^{\phi^*T\MTwo}} \big(\Cof d\phi \big)}{V}_{\phi^*T\MTwo} \dVolOne.
\]

Denote by $\psi : \MOne \times I \to \MTwo$ the map $\psi(p,t) = \phi_t\left( p \right)$. Let $P:\MOne \times I \to \MOne$ be the projection $P(p,t)=p$. Consider the following two vector bundles over $\MOne \times I$:
(i) $\brk{P^*\brk{T\MOne}}^* \cong P^{*}\brk{T^*\MOne}$, whose fiber over $(p,t)$ is $T^*_p\MOne$, and 
(ii) $\psi^{*}\brk{T\MTwo}$, whose fiber over $(p,t)$ is $T_{\phi_t(p)}\MTwo$.

Note that $(d\phi_t)_p:T_p\MOne \to T_{\phi_t(p)}\MTwo$, i.e., $(d\phi_t)_p \in T^*_p\MOne \otimes T_{\phi_t(p)}\MTwo$.
Running over all the pairs $(p,t) \in \MOne \times I$ we obtain a section of the vector bundle $W= \brk{P^{*}\brk{T\MOne}}^* \otimes \psi^{*}\brk{T\MTwo}$.

Now,
\beq 
\label{eq:energy_time_deriv_Riemann}
\begin{split}
\left. \deriv{}{t} E( \phi_t) \right|_{t=0}=&\int_{\MOne} \left. \deriv{}{t}  \Det (d\phi_t) \right|_{t=0}   \dVolOne \\
&= 
\int_{\MOne} \left.    \IP{\Cof (d\phi_t)}{\nabla_{\pd{}{t}}^W d\phi_t}_{P^{*}\brk{T\MOne},\psi^{*}\brk{T\MTwo}} \right|_{t=0} \dVolOne \\ 
&\quad +
  \int_{\MOne}   \IP{\Cof d\phi}{\left. \nabla^W_{\pd{}{t}} d\phi_t   \right|_{t=0} }_{T\MOne,\phi^*T\MTwo} \dVolOne ,
\end{split}
\eeq
where the second equality follows from an application of \lemref{lem:Cofactor_grad_Determinant_bundle} (with $A=d\phi_t$ and $V=\partial/\partial t$).

It is well-known that 
\beq 
\label{eq:symmetric_equality}
\left. \nabla^W_{\pd{}{t}} d\phi_t  \right|_{t=0} = \nabla^{\phi^*T\MTwo}V,
\eeq 
(see e.g.~\cite[Proposition 2.4, Pg 14]{EL83}).
Eqs. \eqref{eq:energy_time_deriv_Riemann} and \eqref{eq:symmetric_equality} imply
\[
\begin{split}
\left. \deriv{}{t} E\left( \phi_t \right) \right|_{t=0} &= \int_{\M}   \IP{\Cof d\phi}{\nabla^{\phi^*T\MTwo}V }_{T\MOne,\phi^*T\MTwo} \dVolOne \\ &=\int_{\M} \IP{\delta_{\nabla^{\phi^*T\MTwo}} \big(\Cof d\phi \big)}{V}_{\phi^*T\MTwo} \dVolOne,
\end{split}
\]
where the last equality follows from \defref{def:coderivative} of the coderivative.
\end{proof}

Two comments are in order: (i)
The identity \eqref{eq:symmetric_equality} is fundamental in the computation of variations between manifolds.  It is proved in  \cite[Proposition 2.4]{EL83}, relying on the symmetry of the connection on $T\MTwo$; it does not require metricity.
(ii)  The application of \lemref{lem:Cofactor_grad_Determinant_bundle} in \eqref{eq:energy_time_deriv_Riemann} requires the  connections on  $P^{*}\brk{T^*\MOne}$ and $\psi^{*}\brk{T\MTwo}$ to be metrically-compatible. Since the Levi-Civita connections on $T\MOne$ and $T\MTwo$ are metrically-compatible, so are all their induced connections.
Thus, both the metricity and the symmetry of the Levi-Civita connection on $T\MTwo$ were used in the proof of \lemref{lem:E-l_det}.



\paragraph{Acknowledgements}
We are grateful to Giuseppe Zurlo, Amitai Yuval and Deane Yang for enlightening discussions. 
This research was partially funded by the Israel Science Foundation (Grant No. 1035/17), and by a grant from the Ministry of Science, Technology and Space, Israel and the Russian Foundation for Basic Research, the Russian Federation.

\appendix
%
%
%

\section{Proof of Lemma~\ref{lem:Cofactor_grad_Determinant_bundle}}
\label{app:deriv_det_bundle}

\lemref{lem:Cofactor_grad_Determinant_bundle} is concerned with the differentiation of the determinant of a bundle morphism between vector bundles. Since the intrinsic definition of the determinant (\ref{def:intrinsic_det}) involves the Hodge-dual, we will need Identity~\ref{eq:Hodge_star_cov_diff_commute} below regarding the behavior of the Hodge operator with respect to covariant differentiation. 

Let $\M$ be a smooth $d$-dimensional manifold.
Let $E$ be an oriented vector bundle over $\M$ (of arbitrary finite rank $n$), endowed with a Riemannian metric $\h$ and a metrically-compatible affine connection $\nabla^E$. 
Note that
$\nabla^E$ induces a connection on $\Lambda_k(E)$ (also denoted by $\nabla^E$); this induced connection is compatible with the metric induced on $\Lambda_k(E)$ by $\h$. Denote by $\star^k_E$ the fiber-wise Hodge-dual $\Lambda_k(E)\to \Lambda_{n-k}(E)$ (which is induced by the orientation on $E$ and $\h$).
Then, 
\beq
\label{eq:Hodge_star_cov_diff_commute}
\star^k_{E} (\nabla^E_X  \beta)=\nabla^E_X (\star^k_{E} \beta)
\eeq
for every $\beta \in \Gamma(\Lambda_k(E))$ and $X \in \Gamma(T\M)$.

This follows from the fact $\star^k_E$ is consistent with the metric, hence it is parallel with respect to  metrically-compatible connections. Indeed,
\[
\nabla^E_X (\star^k_{E} \beta)=(\nabla_X   \star^k_{E}) \beta+  \star^k_{E} (\nabla^E_X  \beta)=\star^k_{E} (\nabla^E_X  \beta).
\]


\emph{Proof of Lemma~\ref{lem:Cofactor_grad_Determinant_bundle}:}
   Let $e_1,\dots,e_d$ be a positive orthonormal frame of $E$. 
  \[
    \Det(A)=   \star^d_{F} \circ \extp^d A \circ \star^0_{E}(1)=   \star^d_W \extp^d A \big( e_1 \wedge \dots \wedge e_d \big)= \star^d_{F} \big( Ae_1 \wedge \dots \wedge Ae_d \big)
  \]

  Using the Leibniz rule for the wedge product, we get
\beq
\label{eq:eq_deriv_det}
\begin{split}
  V\Det A &= V \star^d_{F}\big( Ae_1 \wedge \dots \wedge Ae_d \big) \stackrel{(1)}{=}   \star^d_{F}  \nabla_V \big( Ae_1 \wedge \dots \wedge Ae_d \big) \\
 	&=  \star^d_{F}  \sum_{i=1}^d  Ae_1 \wedge \dots \wedge \nabla_V (Ae_i) \wedge \dots \wedge Ae_d   \\
	&= \star^d_{F}  \sum_{i=1}^d  Ae_1 \wedge \dots \wedge (\nabla_V A)e_i \wedge \dots \wedge Ae_d  + \star^d_{F}  \sum_{i=1}^d  Ae_1 \wedge \dots \wedge  A(\nabla_{V}e_i) \wedge \dots \wedge Ae_d,
\end{split}
\eeq
Where equality $(1)$ follows from \eqref{eq:Hodge_star_cov_diff_commute}. (Here we used the metricity of the connection on $F$). 

Analyzing the second summand, we get
\[
 \star^d_{F} \extp^d A( \sum_{i=1}^d  e_1 \wedge \dots \wedge \nabla_Ve_i \wedge \dots \wedge e_d)=
 \star^d_{F} \extp^d A\big( \nabla_V  (e_1 \wedge \dots \wedge e_i \wedge \dots \wedge e_d) \big)=0,
 \]
where in the last equality we used the metricity of the connection on $E$.
After eliminating the second summand,  \eqref{eq:eq_deriv_det} becomes
\[
\begin{split}
V\Det A &=  \sum_{i=1}^d \star^d_F (-1)^{i-1} \big( (\nabla_V A)e_i \wedge A e_1  \wedge \dots \wedge \widehat Ae_i \wedge \dots \wedge Ae_d \big) \\
 	&  =  (-1)^{d-1} \sum_{i=1}^d (-1)^{i-1} \star^d_F  \big( (\nabla_V A)e_i \wedge  \star^1_F \star^{d-1}_F ( A e_1 \wedge \dots \wedge \widehat Ae_i \wedge \dots \wedge Ae_d) \big)  \\
	 &   = (-1)^{d-1} \sum_{i=1}^d (-1)^{i-1} \IP{(\nabla_V A)e_i}{ \star^{d-1}_F ( A e_1 \wedge \dots \wedge \widehat Ae_i \wedge \dots \wedge Ae_d)}_F  \\ 
	 &= (-1)^{d-1} \sum_{i=1}^d \IP{(\nabla_V A)e_i}{ \star^{d-1}_F \big( \extp^{d-1} A ( \star_V^1 e_i )\big)}_F\\
 	&= \sum_{i=1}^d \IP{(\nabla_V A)e_i}{  \Cof A(e_i) }_F 	= \IP{\Cof A}{\nabla_V A}_{E,F}.
  \end{split}
  \]
\hfill\ding{110}

\section{Coordinate representation of the Riemannian Piola identity}
\label{app:coordinates}

For completeness, we formulate the Riemannian Piola identity in local coordinates:
Let Roman indices $i,j,k$ denote coordinates on $(\MOne,\g)$ and Greek indices $\alpha,\beta,\gamma$ denote coordinates on $(\MTwo,\h)$.
The functions $\g_{ij}$ and $\h_{\alpha\beta}$ are the entries of the metrics $\g$ and $\h$, respectively, and $\Gamma^\alpha_{\beta\gamma}$ are the Christoffel symbols of $\nabla^{\MTwo}$.
The coordinate representation of the differential $df$ is $\pl_i f^\alpha$; similarly, $\Cof df=(\Cof df)_i^\alpha$, that is
\[
\Cof df=(\Cof df)_i^\alpha \, dx^i \otimes f^*\pl_{\alpha}.
\]
Let $\xi=\xi^\alpha\, f^*\partial_\alpha  \in\Gamma(f^*T\MTwo)$ be a compactly-supported section. In local coordinates
\[
\nabla^{f^*T\MTwo} \xi=(\nabla^{f^*T\MTwo} \xi)_j^\beta \, dx^j \otimes f^*\pl_{\beta},
\]
where
\[
(\nabla^{f^*T\MTwo} \xi)_j^\beta=\pl_j \xi^\beta +  (f^*\Gamma^\beta_{\gamma\delta}) \pl_j f^\gamma  \xi^\delta.
\]
Equation \eqref{eq:Cof_d} reads
\beq
\label{eq:Cof_d__coordinates}
\int_{\MOne} \g^{ij} (f^*\h_{\alpha\beta})\, (\Cof df)_i^\alpha \, \brk{\pl_j \xi^\beta +(f^*\Gamma^\beta_{\gamma\delta}) \pl_j f^\gamma  \xi^\delta} |\g|^{1/2}\,dx = 0,
\eeq

where $|\g|$ is the determinant of the matrix representing the metric $\g$.

We first note that

\[
\g^{ij} (f^*\h_{\alpha\beta})\, (\Cof df)_i^\alpha  = ((\Cof df)^T)^j_\beta,
\]
and that
\[
\Det df = \frac{|f^*\h|^{1/2}}{|\g|^{1/2}} \det [df],
\]
where $[df]$ stands here for the matrix whose entries are $\partial_i f^\alpha$. 

Hence, the coordinate representation of the Laplace expansion \eqref{eq:cof_det_expansion} is
\[
\g^{ij} (f^*\h_{\alpha\beta})\, (\Cof df)_i^\alpha \pl_j f^\gamma = \frac{|f^*\h|^{1/2}}{|\g|^{1/2}} \det [df]\, \delta^\gamma_\beta.
\]

Secondly, 
\beq
\label{eq:eq_coord_middle_1}
\g^{ij} (f^*\h_{\alpha\beta})\, (\Cof df)_i^\alpha = \frac{|f^*\h|^{1/2}}{|\g|^{1/2}} (\cof [df]^T)^j_\beta, 
\eeq
where, as before, $\cof$ denote the matrix-cofactor. Thus, \eqref{eq:Cof_d__coordinates} reads
\[
\int (\cof [df]^T)^j_\gamma \pl_j \xi^\gamma  |f^*\h|^{1/2} \, dx 
+ \int  (f^*\Gamma^\beta_{\beta\gamma}) \det [df]  \xi^\gamma |f^*\h|^{1/2} \, dx  = 0.
\]

Integrating by parts, and since this equation holds for every vector field $\xi$, 
\beq
\label{eq:Cof_d__coordinates_final}
 - \frac{1}{|f^*\h|^{1/2}} \partial_j\brk{(\cof [df]^T)^j_\gamma \  |f^*\h|^{1/2}}  
+ (f^*\Gamma^\beta_{\beta\gamma}) \det [df]     = 0.
\eeq
Note that the metric and the connection on the source manifold do not appear in this equation. 

In \cite[p. 117, bottom]{MH83}, the authors give the following coordinate expression for the Riemannian Piola identity, 
\[
\pl_j \brk{|f^*\h|^{1/2}\,  (\cof [df]^T)_j^\delta}=0,
\]
which lacks the  connection term; hence is invalid even in the Euclidean setting, if choosing a coordinate system for which $\Gamma^\beta_{\beta\gamma}\ne 0$.

Eq.~\eqref{eq:Cof_d__coordinates_final} can be further simplified: using the classical identity
\[
f^*\Gamma^\beta_{\beta\gamma}=\frac{1}{ |f^*\h|^{1/2}}f^*\pl_{\gamma}  \brk{ |\h|^{1/2}},
\]
\eqref{eq:Cof_d__coordinates_final} reduces to 
\beq
\partial_j\brk{(\cof [df]^T)^j_\gamma } =0,
\label{eq:surprising}
\eeq
Perhaps surprisingly, the coordinate representation of $f$ satisfies a Piola identity that makes no reference to the Riemannian structures of $\MOne$ and $\MTwo$. At a second thought, this is not a surprise if we recall that the ``proof by calculation" of the Euclidean Piola identity boils down to the commutation of mixed derivatives, which is satisfied by the local representation of a twice-differentiable function regardless of any metric structure. Note, however, that in order to attribute to \eqref{eq:surprising} an intrinsic meaning one has to revert to the form \eqref{eq:Cof_d__coordinates_final}.

{\footnotesize
\newcommand{\etalchar}[1]{$^{#1}$}
\providecommand{\bysame}{\leavevmode\hbox to3em{\hrulefill}\thinspace}
\providecommand{\MR}{\relax\ifhmode\unskip\space\fi MR }
\providecommand{\MRhref}[2]{%
  \href{http://www.ams.org/mathscinet-getitem?mr=#1}{#2}
}

%

\end{document}